\documentclass[11pt,a4paper]{amsart}

\usepackage{amsthm, amsfonts, amssymb, amsmath, latexsym, enumerate, array}
\usepackage[all]{xy}
\usepackage[latin1]{inputenc}
\usepackage{hyperref,cite,mdwlist,mathrsfs}

\usepackage{anysize}
\marginsize{3.0cm}{3.0cm}{4cm}{4cm}

\newtheorem{thm}{Theorem}
\newtheorem{lem}[thm]{Lemma}
\newtheorem{corol}[thm]{Corollary}
\newtheorem{prop}[thm]{Proposition}
\newtheorem*{thm*}{Theorem}

\theoremstyle{definition}

\newcommand{\sO}{\mathscr{O}}
\newcommand{\sD}{\mathscr{D}}
\newcommand{\sU}{\mathscr{U}}
\newcommand{\sQ}{\mathscr{Q}}
\newcommand{\sV}{\mathscr{V}}

\newcommand{\cA}{\mathcal{A}}
\newcommand{\cB}{\mathcal{B}}
\newcommand{\cC}{\mathcal{C}}
\newcommand{\cE}{\mathcal{E}}
\newcommand{\cH}{\mathcal{H}}
\newcommand{\cS}{\mathcal{S}}
\newcommand{\cT}{\mathcal{T}}
\newcommand{\cK}{\mathcal{K}}
\newcommand{\cL}{\mathcal{L}}
\newcommand{\cQ}{\mathcal{Q}}

\DeclareMathOperator{\Ext}{Ext}

\DeclareMathOperator{\Db}{D^b}

\newcommand{\ZZ}{\mathbb Z} \newcommand{\CC}{\mathbb C}
\newcommand{\NN}{\mathbb N} \newcommand{\PP}{\mathbb P}
\newcommand{\OO}{\mathbb O}

\newcommand{\bR}{\mathbf{R}}
\newcommand{\bL}{\mathbf{L}}

\DeclareMathOperator{\ts}{\otimes}

\newcommand{\xr}{\xrightarrow}
\def\a{\alpha}\def\om{\omega}
\def\gr{\mathrm{gr}}

\allowdisplaybreaks[1]
\begin{document}

\sloppy


\title{On the derived category of the Cayley plane II}

\author{Daniele Faenzi}
\email{{\tt daniele.faenzi@univ-pau.fr}}
\address{Universit\'e de Pau et des Pays de l'Adour \\
 Avenue de l'Universit\'e - BP 576 - 64012 PAU Cedex - France}
\urladdr{{\url{http://univ-pau.fr/~faenzi/}}}

\author{Laurent Manivel}
\email{{\tt laurent.manivel@ujf-grenoble.fr}}
\address{Universit\'e de Grenoble \\
  BP 74, 38402 Saint-Martin d'H\`eres, France}
\urladdr{{\url{http://www-fourier.ujf-grenoble.fr/~manivel/}}}

\thanks{D.F. partially supported by ANR projects INTERLOW and
  GEOLMI}


\begin{abstract}
We find a full strongly exceptional collection for the
Cayley plane $\OO\PP^2$, the simplest rational homogeneous space of
the exceptional group $E_6$.
This collection, closely related to the
one given by the second author in \cite{manivel:cayley},
consists of $27$ vector bundles which are homogeneous for the
group $E_6$, and is a Lefschetz collection with respect to the minimal
equivariant embedding of $\OO\PP^2$. 
\end{abstract}

\maketitle

\section*{Introduction}

\subsection*{The Cayley plane}
Let $X=\OO \PP^2$ be the Cayley plane. This is the closed orbit in the
projectivization of the minimal representation of
the complex simply connected exceptional Lie group $E_6$.
The Cayley plane can also be identified with the quotient $E_6/P_1$, where
$P_1$ is the parabolic subgroup of $E_6$ corresponding to the root
$\alpha_1$ of $E_6$.
This is sketched in the diagram below, where
the fundamental roots $\a_1\ldots,\a_6$ of $E_6$ are depicted.
We will denote by $\om_1,\ldots,\om_6$ the corresponding fundamental weights.

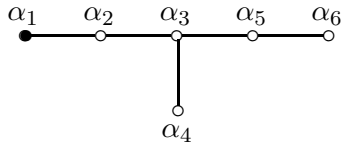
\begin{center} 
\setlength{\unitlength}{5mm}
\begin{picture}(20,5)(-6.5,0) 
\multiput(-.34,3.8)(2,0){5}{$\circ$}
\multiput(0,4)(2,0){4}{\line(1,0){1.7}}
\put(3.7,1.8){$\circ$}
\put(-.3,3.8){$\bullet$}
\put(-.6,4.4){$\a_1$}
\put(1.4,4.4){$\a_2$}
\put(3.4,4.4){$\a_3$}
\put(5.4,4.4){$\a_5$}
\put(7.4,4.4){$\a_6$}
\put(3.45,1.3){$\a_4$}
\put(3.89,2.15){\line(0,1){1.73}}
\end{picture}
\end{center}

The aim of this note is to provide a full strongly exceptional
collection of the derived category of coherent sheaves on the Cayley plane.
This completes \cite{manivel:cayley}, where a strongly exceptional
collection closely related to ours was found. But it was not 
proved that this collection generates the whole derived category. 

Our proof of this missing point has two important ingredients. 
On the one hand, we use specific tensor relations between 
certain homogeneous bundles on the Cayley plane, to show that some 
of these bundles belong to the category generated by our strongly
exceptional collection.
This will be the done in Section \ref{1}.
This allows us, on the other hand, to 
use restriction to certain specific subvarieties of the Cayley
plane, its maximal quadrics, for which we know a full 
exceptional collection of the derived category. 
We will do this in Section \ref{2}.
Section \ref{3} contains the proof of our result.

\subsection*{The octonionic plane geometry}
These maximal, 8-dimensional quadrics are called $\OO$-lines 
since they are copies of $\OO \PP^1$. They define a (slightly degenerate) plane projective 
geometry over the Cayley octonions. Let us recall briefly the 
ingredients from this geometry we will use in this note. For more
details see \cite{iliev-manivel:78}. 

The Cayley plane $\OO\PP^2\subset\PP V_{\om_1}$ has a twin 
$\OO\check{\PP}^2\subset\PP V_{\om_6}=\PP V_{\om_1}^*$. Both 
varieties are isomorphic as embedded projective varieties,
but not equivariantly. If $x$ is a point of $\OO\PP^2$, 
the orthogonal to its tangent space in $\PP V_{\om_1}^*$ cuts
the dual Cayley plane $\OO\check{\PP}^2$ along a smooth
8-dimensional quadric $\check{Q}_x$. Symmetrically, each 
point $\ell$ in $\OO\check{\PP}^2$ defines a smooth  
8-dimensional quadric $Q_\ell$ in the Cayley plane.
 
These quadrics are called $\OO$-lines because generically, 
they have the characteristic properties of 
a plane projective geometry: two general 
$\OO$-lines meet in a single point, and through two general 
points pass a unique $\OO$-line. They can also be considered
as {\it entry-loci} in the following sense. The secant 
variety of $\OO\PP^2$ is a cubic hypersurface in $\PP V_{\om_1}$
which we call the Cartan cubic \cite{iliev-manivel:78}. This 
secant variety is degenerate (and therefore equal 
to the tangent variety of the Cayley plane): this means that 
a general point $y$ of the Cartan cubic belongs to infinitely
many secant lines to the Cayley plane; in fact, if $y$ itself
does not belong to $\OO\PP^2$, the intersection locus 
of these secants with the Cayley plane, traditionally called 
the entry-locus of $y$, is a smooth 8-dimensional 
quadric, 
and in fact an $\OO$-line.

\subsection*{The main result}
To state our main result, we need to setup some material.
The category of homogeneous bundles on the Cayley plane is equivalent
to the category of $P_1$-modules. Let $P_1=L P^u$ be a Levi
decomposition, where $P^u$ is the unipotent radical and $L$ is
reductive.
The center of the Levi factor $L$ is one-dimensional. The quotient 
of $L$ by its center is the semisimple part of $P_1$, and is
isomorphic to $\rm{Spin}_{10}$. 

The unipotent radical $P^u$ has to act trivially on an 
irreducible $P_1$-module, which is therefore completely 
determined by its $L$-module structure.  
An irreducible $L$-module is determined by its highest weight, 
a $L$-dominant weight of $E_6$. The set $L$-dominant weights is 
the set of linear combinations $\om = a_1\om_1 +\cdots + a_6 \om_6$, 
with $a_2,\ldots,a_6 \in \NN$ and $a_1 \in\ZZ$.
We denote by $\cE_\om$ the homogeneous bundle on $X$ corresponding to
the irreducible $L$-module determined by the weight $\om$.

Set $\cS = \cE_{\om_6}$ and $\cS_2 = \cE_{2\om_6}$.
By the Borel-Weil theorem $H^0(X,\cS) \cong V_{\om_6}^* \cong V_{\om_1}$ and
$\cS$ is generated by its global section. In particular for any
point $x$ in the Cayley plane the fibre $\cS_x^*$ of the dual 
bundle can be considered as a subspace of $V_{\om_6}=V_{\om_1}^*$,
and this subspace is nothing else than the linear span of the 
$\OO$-line $\check{Q}_x$. 
 
Recall from \cite{manivel:cayley} that $\cS^* \cong \cS(-1)$ 
and $\cS_2^* \cong \cS_2(-2)$.
Consider the collections:
\begin{align*}
   & \cA = (\cS_2^*, \cS^*, \sO_X), \\
   & \cC = (\cS^*, \sO_X).
\end{align*}

\begin{thm*}
  The Cayley plane admits the following full strongly exceptional
  collection:
  \[
  \Db(\OO\PP^2)=\langle \cA, \cA(1), \cA(2), \cC(3), \cC(4),\ldots, \cC(11)\rangle.
  \]
\end{thm*}

This is a Lefschetz collection in the sense of \cite{kuznetsov:hpd}. 
It was used in \cite{iliev-manivel:78} to study the derived categories 
of 7-dimensional cubics. 

\subsection*{Acknowledgements}

The first author would like to thank Institut Fourier at Grenoble for
the warm hospitality during his visit at the origin of this work.

\section{Generating new bundles from our exceptional collection}
\label{1}

Given the exceptional collection appearing in our main result, we
denote by $\sD$ the full subcategory of $\Db(X)$ generated by it:
\[
\sD =  \langle \cA, \cA(1), \cA(2), \cC(3), \cC(4),\ldots, \cC(11) \rangle.
\]
One may ask, what are the sheaves on $X$ that do lie in $\sD$?
Of course, the goal of this note is to answer: {\it all of them}.
In this section, we will first prove that some particular sheaves lie
in $\sD$.
We will use these sheaves further on, to prove that the orthogonal of 
$\sD$ in $\Db(X)$ is zero, which amounts to our main result.

One of our main tools for getting new bundles from our generators of
$\sD$ will be the following 
classical observation. Given a $P_1$-module $E$, we can consider it as an $L$-module,
and then define a new $P_1$-module structure by extending trivially to 
$P^u$.  We get the associated graded module $\gr(E)$, which is a direct 
sum of irreducible $P_1$-modules. Moreover, $E$ can be reconstructed from
$\gr(E)$ by a series of extensions. 

If we consider an irreducible bundle $\cE_\om$ defined by an
$L$-dominant weight $\om$, the corresponding irreducible 
$\mathrm{Spin}_{10}$-representation has highest weight 
$\om = a_2 \om_2 +\cdots + a_6 \om_6$. Relations between representations
of $\mathrm{Spin}_{10}$ will therefore imply relations between
the corresponding homogeneous bundles. For example, the following
fact is general: if we consider a fundamental representation 
corresponding to an extremal node of some Dynkin diagram, its 
second skew-symmetric power contains the fundamental representation 
corresponding to the neighbouring node, and so on till we attain 
a triple node or a multiple edge. For $\mathrm{Spin}_{10}$ this 
is particularly neat since these inclusions are in fact always 
identities. In particular the fundamental representation corresponding
to the triple node can be written as a wedge power in three 
different ways. To get the correct relations between $L$-modules 
there remains to adjust the action of the one dimensional center, which 
can be done by computing the first Chern class. Once we cross a 
triple node, the wedge power contains (and will be equal in the 
cases we will be interested in) the irreducible representation whose
highest weight is the sum of the two fundamental weights beyond the 
triple node (this follows from Theorem 2.1 in \cite{landsberg-manivel:series}).
We get:

\begin{lem}\label{wedges}
The homogeneous bundles defined by the fundamental representations of $L$ can be described as:
 \begin{align*}
 &\cE_{\om_6}\cong\cS, && \cE_{\om_5}\cong\wedge^2\cS, && \cE_{\om_3}\cong\wedge^3\cS, \\
 &\cE_{\om_4}\cong\cT_X, && \cE_{\om_2}\cong\Omega_X(2), && \cE_{\om_3}\cong\wedge^2 \cT_X\cong\Omega^2_X(3).
 \end{align*}
Moreover, the wedge powers that come next are:
$$\wedge^4\cS\cong\cE_{\om_2+\om_4}, \qquad \wedge^3 \cT_X\cong\cE_{\om_2+\om_5}, \qquad \Omega^3_X(4)\cong\cE_{\om_4+\om_5}.$$
\end{lem}

We already mentioned that we also have $\cE_{\om_1}\cong\sO_X(1)$, the very ample line bundle which
defines the natural embedding of $X$ into $\PP V_{\om_1}$.  The twisted bundle
 $\cS(1)$ is the normal bundle of this embedding.
The restricted Euler sequence reads:
\begin{align}
\label{euler} & 0 \to \Omega_{\PP V_{\om_1} | X} \to V_{\om_6} \ts \sO_X(-1) \to \sO_X\to 0, \\ 
\intertext{while the conormal bundle sequence is:}
\label{normal} & 0 \to \cS^*(-1) \to \Omega_{\PP V_{\om_1} | X} \to \Omega_X \to 0.
\end{align}
We deduce that $\Omega_X$ and $\cT_X$ are the middle cohomology,
respectively, of the complexes:
\begin{align}
  \label{cotangent}
  & 0 \to \cS^*(-1) \to V_{\om_6} \ts \sO_X(-1) \to \sO_X \to 0, \\
  \label{tangent}
 & 0 \to \sO_X \to V_{\om_1} \ts \sO_X(1) \to \cS^*(2)  \to 0.
\end{align}

In the next lemmas we deal with some twisted symmetric and exterior powers of $\cS$.

\begin{lem} \label{6term}
  We have an exact sequence:
  \[
  \mbox{\small $
    0 \to \cS_2^* \to V_{\om_6} \ts \cS^* \to  (V_{\om_1} \oplus
  V_{\om_5}) \ts \sO_X \to (V_{\om_6} \oplus V_{\om_2}) \ts \sO_X(1) \to V_{\om_1} \ts \cS^*(2) \to \cS_2^*(3) \to 0.
  $}\]
  In particular, $\cS_2^*(t)$ lies in $\sD$ for $3 \le t \le 12$.
\end{lem}

\begin{proof}
We start with the exact sequence:
\begin{align*}
   0 \to \cS^* \to V_{\om_6} \ts \sO_X \to \cQ \to 0
\end{align*}
and its symmetric square:
\begin{align*}
 0 \to S^2\cS^* \to V_{\om_6} \ts \cS^*\to \wedge^2V_{\om_6} \ts \sO_X\to \wedge^2\cQ \to 0.
\end{align*}
Consider the composition:
\begin{align*}
 S^2\cS^* \to V_{\om_6} \ts \cS^*\to S^2V_{\om_6} \ts \sO_X\to V_{\om_6}^* \ts \sO_X,
\end{align*}
where the rightmost morphism is defined by the polarization of the 
invariant cubic form on $V_{\om_1}$. We claim that the image 
 of this 
composition is the one dimensional factor in the decomposition $S^2\cS^*=\cS_2^*\oplus \sO_X(-1)$. 
Indeed, let $c$ be an equation of the Cartan cubic in $\PP V_{\om_1}$; we denote its polarisation
in the same way. We know that the (dual) 
Cayley plane is the singular locus of this hypersurface, and that it meets $\PP(\cS_x^*)$, for any
$x\in\OO\PP^2$, along the quadric $\check{Q}_x$. 
This means that we must have an identity of the form
\[
c(s,s,v)=\ell(v)q(s) \qquad \forall s\in  \cS_x^* \subset V_{\om_6},
\quad \forall v\in V_{\om_6},
\]
where $q$ is an equation of $\check{Q}_x$, and $\ell$ a linear form. This implies that the kernel
of the map  $S^2\cS_x^*\to V_{\om_6}^*$ is the hyperplane defined by $q$, which is by \cite{manivel:cayley}
the fibre at $x$ of the bundle $\cS_2^*$. 

We thus get a complex:
\begin{align}\label{one-half}
 0 \to \cS_2^* \to V_{\om_6} \ts \cS^*\to (\wedge^2 V_{\om_6}\oplus V_{\om_1}) \ts \sO_X\to \cK \to 0,
\end{align}
where the vector bundle $\cK$ fits into the exact sequence:
\begin{align}\label{suiteK}
 0 \to  V_{\om_1}/ \sO_X(-1) \to \cK \to \wedge^2\cQ \to 0,
\end{align}
where $V_{\om_1}/ \sO_X(-1)$ is the restriction to $X$ of $\cT_{\PP V_{\om_1}}(-1)$.
Note that the rightmost term of this sequence is itself an extension:
\begin{align*}
 0 \to  \cT_X(-1) \to V_{\om_1}/ \sO_X(-1) \to \cS \to 0.
\end{align*}
In order to prove the lemma there just remains to check that $\cK$ is isomorphic to 
$\cK^*(1)$: this will allow to fit the exact sequence \ref{one-half} and its twisted 
transpose into a single long exact sequence, which is the desired complex since 
$\wedge^2V_{\om_1} \cong V_{\om_2}$, $\wedge^2V_{\om_6} \cong V_{\om_5}$, and
$V_{\om_1}^* \cong V_{\om_6}$, $V_{\om_2}^* \cong V_{\om_5}$.

Recall that $\cQ$ fits into an exact sequence:
\begin{align*}
 0 \to  \Omega_X(1) \to \cQ \to  \sO_X(1) \to 0,
\end{align*}
from which we get the exact sequence:
\begin{align*}
 0 \to  \Omega^2_X(2) \to \wedge^2\cQ \to  \Omega_X(2) \to 0.
\end{align*}
Putting this exact sequence together with (\ref{suiteK}) and the 
exact sequence that follows, we deduce  a complex:
\begin{align*}
 0 \to  \cT_X(-1) \to \cK \to  \Omega_X(2) \to 0
\end{align*}
whose middle cohomology is a vector bundle $\cL$ fitting into the
exact sequence: 
 \begin{align*}
 0 \to  \cS \to \cL \to  \Omega^2_X(2) \to 0.
\end{align*}
We claim that this extension is trivial. Indeed, all the construction
being equivariant with respect to the simply connected group $E_6$,
it must be defined by an invariant class $e\in \Ext^1_X(\Omega^2_X(2),\cS)^{E_6}$.
But a straightforward application of Bott's theorem implies that there
is no non-zero such class. Hence $\cL= \cS \oplus \Omega^2_X(2)$. 
Now, using the same argument we check that there is no non-trivial 
extension of $\cL$ by  $\cT_X(-1)$, and finally, no non-trivial 
extension of $\cL\oplus \cT_X(-1)$ by $\Omega^2_X(2)$. This implies that
$\cK$ is completely reducible. More precisely,
 \begin{align*}
 \cK\simeq \cS \oplus \Omega^2_X(2)\oplus \cT_X(-1)\oplus \Omega^2_X(2)\simeq \cK^*(1).
\end{align*}
Indeed, we already know that $\cS\simeq \cS^*(1)$, and $\Omega^2_X(2)\simeq \wedge^2\cT_X(-1)$
by Lemma \ref{wedges}. 
\end{proof}

\begin{lem}
  \label{wedge2}
  The bundle $\wedge^2 \cS=\wedge^2 \cS^*(2)$ belongs to $\sD_2=\langle \cS^*,\sO_X,\sO_X(1),\sO_X(2),\cS^*(3) \rangle.$
  In particular,
  \[\wedge^2 \cS^*(t) \in \sD \qquad \mbox{for $2\le t \le 10$}.\]
\end{lem}

\begin{proof}
 Consider the adjoint representation $V_{\om_4}$. By computing its restriction to 
  $\mathrm{Spin}_{10}$ and adjusting the first Chern class, we obtain:
  \begin{equation}
    \label{adjoint}
    \gr(V_{\om_4} \ts \sO_X) \cong \cT_X \oplus \Omega_X \oplus \wedge^2
    \cS(-1) \oplus \sO_X.
  \end{equation}

  Now, using \eqref{cotangent} and \eqref{tangent} we get 
  that $\cT_X$ belongs to $\langle \sO_X,\sO_X(1),\cS^*(2) \rangle$
  and
  $\Omega_X$ to $\langle \cS^*(-1),\sO_X(-1),\sO_X \rangle$, so that 
  \eqref{adjoint} ensures that 
  $\wedge^2 \cS(-1) \cong \wedge^2 \cS^*(1)$ lies in $\langle
  \cS^*(-1),\sO_X(-1),\sO_X,\sO_X(1),\cS^*(2) \rangle$.
  This clearly implies our statement.
\end{proof}

\begin{lem} \label{wedge3}
The bundle 
    $\wedge^3 \cS \cong \wedge^3 \cS^*(3)$ lies in:
\[\sD_3=\langle
\sO_X,\cS_2^*(1),\cS^*(1),\sO_X(1),\cS^*(2),\sO_X(2),\sO_X(3)
\rangle.
\]
In particular, 
  \[ \wedge^3 \cS^*(t) \in \sD \qquad \mbox{for $3 \le t \le 11$}.\]
\end{lem}

\begin{proof}
   We have to check that $\wedge^3 \cS \cong \Omega^2_X(3)$ (by Lemma \ref{wedges})
lies in $\sD_3$.
   To do this, we take the exterior square of the conormal sequence \eqref{normal}, and
   we twist by $\sO_X(3)$, so that we are lead to ask if
   $S^2 \cS^*(1)$,  $\cS^* \ts \Omega_{\PP V_{\om_1} | X}(2)$ and
   $\Omega^2_{\PP V_{\om_1} | X}(3)$ 
   belong to $\sD_3$.

   For the first bundle, recall from \cite{manivel:cayley} the isomorphism $S^2 \cS^*(1) \cong
   \cS_2^*(1) \oplus \sO_X$.
   For the second one, twisting by $\cS^*$ the Euler sequence
   \eqref{euler} we see that $\cS^* \ts \Omega_{\PP V_{\om_1} |
     X}(2)$ belongs to $\langle \cS^*(1),\cS^*(2) \rangle$.
   For the third bundle, note that $\Omega^2_{\PP V_{\om_1} | X}(3)$ lies in $\langle
   \sO_X(1),\sO_X(2),\sO_X(3) \rangle$.
   This shows that $\wedge^3 \cS$ lies in $\sD_3$.
   
   Using the isomorphism $\wedge^3 \cS^* \cong \wedge^3
   \cS(-3)$, and Lemma \ref{6term}, we deduce that $\wedge^3 \cS^*(t)$ lies in $\sD$, for $3 \le t \le 11$.
\end{proof}

\begin{lem} \label{wedge4}
  The bundle  $\wedge^4 \cS \cong \wedge^4 \cS^*(4)$ lies in: 
\[\sD_4=\langle \cS^*(1),\sO_X(1),\cS^*(2),\sO_X(2),\cS_2^*(3),\cS^*(3),\sO_X(3),\cS^*(4) \rangle.\]
  In particular, 
 \[ \wedge^4 \cS^*(t) \in \sD, \qquad \mbox{for $3 \le t \le 11$}.\]
\end{lem}

\begin{proof}
  By Lemma \ref{wedges}, we have $\wedge^4 \cS \cong \cE_{\om_2+\om_4}$. 
    Using the same method as before we observe that the completely irreducible 
bundle associated to $\cE_{\om_2} \ts \cE_{\om_4}\cong \Omega_X(2) \ts \cT_X$ is:
\[\gr(\cE_{\om_2} \ts \cE_{\om_4}) =
  \cE_{\om_2+\om_4} \oplus \wedge^2 \cS(1)\oplus \sO_X(2).\]
  Therefore, it we will be enough to show that the three bundles $\Omega_X(2) \ts \cT_X$,
$\wedge^2 \cS(1)$ and $\sO_X(2)$  lie in the subcategory $\sD_4$.
  This is obviously the case for $\sO_X(2)$, and also for $\wedge^2
  \cS(1)$ in view of Lemma \ref{wedge2}.
  To  show that $\Omega_X(2) \ts \cT_X$
  lies in $\sD_4$, we first tensor the complex \eqref{cotangent}, twisted by $\sO_X(2)$, 
with its dual \eqref{tangent}.
  The proof will be finished once we show that all the terms appearing in the resulting complex 
  lie in $\sD_4$.
  This is clear for  the twists of $\sO_X$ and of $\cS^*$.
  But the only term which is not of this form is: 
\[\cS^* \ts \cS^*(3)\cong \wedge^2 \cS^*(3) \oplus \cS_2^*(3) \oplus
  \sO_X(2),\]
and we have already seen that $\wedge^2 \cS^*(3) \cong \wedge^2 \cS(1)$ lies in $\sD_4$, 
as well as $\cS_2^*(3)$ and $\sO_X(2)$.
\end{proof}

Set $\cS_3 = \cE_{3 \om_1}$ and recall from \cite{manivel:cayley} that
$S^3 \cS \cong \cS_3 \oplus \cS(1)$ and $\cS_3^* \cong \cS_3(-3)$.

\begin{lem} \label{S3}
  The bundle $\cS_3^*(t)$ lies in $\sD$ for $1 \le t \le 6$.
\end{lem}

\begin{proof}
  Taking the symmetric cube of \eqref{normal} twisted by $\sO_X$, we
  get:
  \[
  0 \to S^3 \cS^* \to S^2 \cS^* \ts \Omega_{\PP V_{\om_1} | X}(1) \to
  \cS^* \ts \Omega^2_{\PP V_{\om_1} | X}(2) \to
  \Omega^3_{\PP V_{\om_1} | X}(3) \to \Omega^3_X(3) \to 0.
  \]
  We look at the various terms of this sequence.
  The first one decomposes as $\cS^*(-1) \oplus \cS_3^*$, and our aim
  is to understand which of its twists lie in $\sD$, so let us study
  the other terms.
  The second term lies in $\langle \sO_X(-1),\cS_2^*,\sO_X, \cS_2^*(1) \rangle$,
  the third one in $\langle \cS^*,\cS^*(1),\cS^*(2) \rangle$, and the 
  fourth one in $\langle  \sO_X,\sO_X(1),\sO_X(2),\sO_X(3) \rangle$. So
  it remains to study the last term $\Omega^3_X(3)$.
  
 By Lemma \ref{wedges} we know that $\Omega^3_X(3) \cong \cE_{\om_4+\om_5}(-1)$.
Moreover we check that:
  \begin{align*}
    & \gr(\cE_{\om_4} \ts \cE_{\om_5}(-1)) \cong \cE_{\om_4+\om_5}(-1) \oplus \cE_{\om_2+\om_6}(-1) \oplus \cE_{\om_4}, \\
    & \gr(\cE_{\om_2} \ts \cE_{\om_6}(-1)) \cong \cE_{\om_2+\om_6}(-1) \oplus \cE_{\om_4}.
  \end{align*}
  We deduce that $\Omega^3_X(3)$ lies in the category generated by 
$\cE_{\om_4} \ts \cE_{\om_5}(-1)$, $\cE_{\om_2} \ts \cE_{\om_6}(-1)$ and $\cE_{\om_4}$. 
  Some of these terms are readily
  settled, for instance
  $\cE_{\om_4} \cong \cT_X \in \langle \sO_X,\sO_X(1),\cS^*(2)
  \rangle$ by \eqref{tangent}.
  Next,  $\cE_{\om_6}(-1) \cong \cS^*$ and
  $\cE_{\om_2} \cong \Omega_X(2) \in \langle \cS^*(1), \sO_X(1),\sO_X(2) \rangle$ by
  \eqref{cotangent}, so $\cE_{\om_2} \ts \cE_{\om_6}(-1)$ lies in 
  $\langle \wedge^2 \cS^*(1),\sO_X,\cS_2^*(1) ,\cS^*(1),\cS^*(2) \rangle$,
  so Lemma \ref{wedge2} will take care of this term.
  
  It remains to analyze $\cE_{\om_4} \ts \cE_{\om_5}(-1)$.
  First recall by Lemma \ref{wedges} that this bundle is isomorphic to
  $\cT_X \ts \wedge^2 \cS^*(1)$.
  Moreover, $\cT_X \in \langle \sO_X,\sO_X(1),\cS^*(2) \rangle$,
  and by Lemma \ref{wedge2}, $\wedge^2\cS^*(1)$ belongs to
  $\langle \cS^*(-1), \sO_X(-1), \sO_X, \sO_X(1),\cS^*(2) \rangle$.
  Recalling that $\cS^* \ts \cS^* \cong \cS_2^* \oplus
  \wedge^2 \cS^* \oplus \sO_X(-1)$ we deduce that:
  \[
  \cT_X \ts \wedge^2\cS^*(1) \in \left\langle \wedge^2 \cS^*(1),\cS_2^*(1),
  \{\cS^*(r), \sO_X(s),\mbox{\, for $-1\le r,s \le 3$}\},\cS_2^*(4),
  \wedge^2 \cS^*(4)
  \right\rangle.
  \]
  Putting all these elements together, we get:
  \[
  \cS_3^* \in \left\langle
  \{\cS_2^*(p),\cS^*(r), \sO_X(s),\mbox{\, for $ 0\le p \le 4,-1\le r\le 3,-1\le s \le 3$}\},
  \wedge^2 \cS^*(1),  \wedge^2 \cS^*(4)
  \right\rangle.
  \]
  Using Lemma \ref{wedge2} and Lemma \ref{6term}, we deduce that
  $\cS_3^*(t)$ lies in $\sD$ for $1 \le t \le 6$.
\end{proof}

One can deduce from these lemmas that many more bundles lie in $\sD$.
One example of this is the next result, that we will need further on.

\begin{corol} \label{wedge-tensor}
  We have:
  \begin{align}
    \label{p=2}    & \wedge^2 \cS^* \ts \cS^*(t) \in \sD, && \mbox{for
      $2 \le t \le 9$}, \\    
    \label{p=3}    & \wedge^3 \cS^* \ts \cS^*(t) \in \sD, && \mbox{for
      $5 \le t \le 8$},\\
    \label{p=4}    & \wedge^4 \cS^* \ts \cS^*(t) \in \sD, && \mbox{for
      $5 \le t \le 7$}.
  \end{align}
 \end{corol}
 
\begin{proof}
  In order to prove \eqref{p=2}, first recall that by Lemma \ref{wedge2}, 
we know that $\wedge^2\cS^*(2)\in \langle \cS^*, \sO_X, \sO_X(1), \sO_X(2),\cS^*(3) \rangle$. 
We deduce that:
\[\wedge^2\cS^*\ts \cS^*(4)\in \langle
\wedge^2\cS^*(2), \sO_X(1), \cS_2^*(2), \cS^*(2), \cS^*(3), \cS^*(4),  \sO_X(4), \cS_2^*(5), \wedge^2\cS^*(5)
\rangle.\]
Therefore $\wedge^2\cS^*\ts \cS^*(t)\in \sD$ for $4\le t\le 9$. 

Now we use the relation:
\begin{equation}
  \label{identity}
  \wedge^2\cS^*\ts\cS^*\oplus\cS_3^* \cong \cS_2^*\ts\cS^* \oplus\wedge^3\cS^*,
\end{equation}
which is easy to establish using
$\mathrm{Spin}_{10}$-representations. With Lemmas \ref{S3} and
\ref{wedge3} 
this implies that $\cS_2^*\ts\cS^*(t)\in \sD$ for $4\le t\le 6$. Then we can use 
the complex from Lemma \ref{6term}, twisted by $\cS^*$, and  deduce
that this is still true for $t=2,3$.  
Then $\wedge^2\cS^*\ts\cS^*(t)$ also belongs to $\sD$ for $t=2,3$.

  To prove \eqref{p=3}, we use Lemma \ref{wedge3}, and we tensor
  by $\cS^*(-3)$ the generators of the subcategory $\sD_3$ appearing in
  this Lemma.
  We find that $\wedge^3 \cS^* \ts \cS^*$ lies in the subcategory
  generated by $\cS^*(r)$, with $-3\le r \le 0$ and by $\cS^* \ts \cS^*(-2)$,
  $\cS^* \ts \cS^*(-1)$, $\cS_2^* \ts \cS^*(-2)$.
  All these terms have been previously encountered. 
  Using the relation \eqref{identity}, Lemma \ref{S3}, Lemma \ref{wedge2} and Lemma \ref{6term} we
  get \eqref{p=3}. 

  It remains to prove \eqref{p=4}.
  We use Lemma \ref{wedge4} to see that:
  \[
  \wedge^4 \cS^* \ts \cS^* \in 
  \left\langle
    \wedge^2 \cS^*(q),\cS_2^*(p),\cS^*(r),\sO_X(s)
  \right\rangle,
  \]
  with $-3 \le q,p \le 0$, $-3 \le r \le -1$, and $-4 \le s \le -1$.
  Now the relation \eqref{p=4} can be deduced from Lemmas \ref{6term}, \ref{wedge2} and \ref{S3}.
\end{proof}

\pagebreak

\section{Quadrics of dimension seven and eight in the Cayley plane}
\label{2}

In order to prove our main result, a key idea is to restrict our exceptional collection to the
quadrics of dimensional 7 and 8 contained in $X$, for these
quadrics are related to dependency loci of sections of $\cS$.
Let us first outline this relationship.


Recall that $\cS$ is a globally generated bundle of rank $10$.
One can compute from \cite{iliev-manivel:78} that $c_{10}(\cS)=0$
and $c_1(\cS)=5H$, where $H$ is the hyperplane class of $X$.
So, a general global section $s$ of $\cS$ vanishes nowhere, and we have
a rank-$9$ vector bundle $\cS_s$ on $X$, having $c_1(\cS_s)=5H$, defined
by the sequence: 
\begin{equation}
  \label{F}
0 \to \sO_X \xr{s} \cS \to \cS_s \to 0.  
\end{equation}

\begin{lem} \label{3quadrics}
  Consider a vector bundle $\cS_s$ as above and a global section
  $\sigma$ of $\cS_s$. Then:
  \begin{enumerate}[i)]
  \item \label{i} if the global section $\sigma$ is general enough, then it vanishes on the union
    $\sQ_\sigma$ of three smooth 7-dimensional quadrics,
  \item \label{ii} any smooth 7-dimensional quadric on the Cayley plane is a 
    component of some $\sQ_\sigma$.
  \end{enumerate}
\end{lem}

\begin{proof}
A section $s$ of $\cS$ is defined by a vector $v\in V_{\omega_1}$, 
and $s$ vanishes at $x\in\OO\PP^2$ if and only if $v$ belongs to
the affine tangent space $\hat{T}_x\OO\PP^2$. This can
happen only if $v$ belongs to the union of the tangent spaces to
the Cayley plane, which is nothing else than the cone over the
Cartan cubic. In particular this cannot happen if $v$ is general:
this explains why $c_{10}(\cS)=0$.

A section $\sigma$ of $\cS_s$ is defined by a vector $\bar{u}\in V_{\omega_1}/\CC v$. 
It vanishes at $x\in\OO\PP^2$ if and only if $u$ belongs to
the span of $v$ and the affine tangent space $\hat{T}_x\OO\PP^2$.
Generically, the projective line generated by $u$ and $v$ cuts the Cartan
cubic in three points $y_1, y_2, y_3$, and $\cT_x\OO\PP^2$ must contain
one of these three points. Suppose that this point is $y_1$. Then $x$ 
must be on the entry-locus of $y_1$, a smooth 8-dimensional quadric 
$Q_1$. Moreover $y_1$ must belong to $\cT_x Q_1$, which means that $x$ is on
the polar hyperplane $H_1$ to $y_1$ with respect to $Q_1$. Finally, 
$x$ has to belong to the 7-dimensional quadric $Q_1\cap H_1$.   
The statement \eqref{i} follows.

To check \eqref{ii}, let $Q$ be a smooth 7-dimensional quadric. Observe that
its linear span $\langle Q\rangle$ cannot be contained in $\OO\PP^2$, 
whose maximal linear spaces are only five dimensional by \cite[Corollary 7.19]{landsberg-manivel:projective}. 
Since the Cayley plane
is cut out by quadrics, this implies that $\OO\PP^2\cap\langle Q\rangle
=Q$. If we choose a point $z$ in $\langle Q\rangle \setminus Q$, its entry locus 
$\tilde{Q}$ is an 8-dimensional quadric containing $Q$, which must therefore
be a hyperplane section of $\tilde{Q}$, say $Q=\tilde{Q}\cap H$. Let $y$ be the 
point of $\langle \tilde{Q}\rangle$ polar to $H$ with respect to $\tilde{Q}$.
Since $Q$ is smooth, $y$ does not belong to $\tilde{Q}$, hence not 
to the Cayley plane either. A general line through $y$ will allow
to define a section  $\sigma$ of some $\cS_s$ such that $Q$ is a component 
of $\sQ_\sigma$.
\end{proof}

The following lemma might be of independent interest, and closely
resembles \cite[Theorem 2.3.2]{okonek-schneider-spindler}.
Let us denote by $Q_n$ a smooth quadric hypersurface in $\PP^{n+1}$.
If $n$ is odd, we denote by $\Sigma$ the spinor bundle on $Q_n$.
If $n$ is even, we denote by $\Sigma_+$ and $\Sigma_-$ the two spinor
bundle on $Q_n$ (see \cite{ottaviani:spinor}).

\begin{lem} \label{Q2}
  Let $n \ge 3$ and let $E$ be a vector bundle on $Q_n$.
  Then $E$ splits as $\oplus_i \sO_{Q_n}(a_i)$ if and only if the
  restriction of $E$ to some quadric surface 
$Q_2 \subset Q_n$ splits as $\oplus_i \sO_{Q_2}(a_i)$.
\end{lem}

The proof closely resembles the argument for $\PP^n$.
We provide the argument for the reader's convenience.  

\begin{proof}
  One direction is obvious, and it suffices to prove that $E$ splits
  as $\oplus_i \sO_{Q_n}(a_i)$ if the restriction $E_{|Q_{n-1}}$ to a 
codimension one quadric splits as
  $\oplus_i \sO_{Q_{n-1}}(a_i)$, for $n \ge 3$.
  Consider the exact sequence:
  \[
  0 \to E(t-1) \to E(t) \to E_{|Q_{n-1}}(t) \to 0.
  \]
  Since $E_{|Q_{n-1}}$ splits as $\oplus_i \sO_{Q_{n-1}}(a_i)$, we
  have $H^k(Q_{n-1},E_{|Q_{n-1}}(t))=0$ for 
  $0<k<n-1$ and for all $t\in \ZZ$.
  Moreover, since $E$ is locally free we must have $H^k(Q_n,E(t))=0$ for
  all $k<n$ and $t \ll 0$ as well as for $k>0$ and $t \gg 0$.
  We deduce that $H^k(Q_n,E(t))=0$ for $1 < k < n-1$ and for all $t\in \ZZ$.
  Note that $H^1(Q_{n-1},E_{|Q_{n-1}}(t))$ vanishes for all $t\in \ZZ$
  because $n\ge 3$, so that,  for any integer $t$, $H^k(Q_n,E(t-1))$ maps onto
  $H^k(Q_n,E(t))$.
  This gives $H^1(Q_{n},E(t))=0$ for all $t\in \ZZ$. 
  Applying the argument to $E^*$ and using Serre duality, one gets
  $H^{n-1}(Q_{n},E(t))=0$ for all $t\in \ZZ$. 

  Therefore,
 by \cite{knorrer:ACM}, $E$ splits as the direct sum of line
  bundles, plus a direct sum of twisted spinor bundles.
  But the restriction to $Q_{n-1}$ of none of the bundles $\Sigma$, $\Sigma_+$,
  $\Sigma_+$ can contain any of the summands $\sO_{Q_{n-1}}(a_i)$
  (see \cite{ottaviani:spinor} for the behavior of the restriction of
  spinor bundles to linear sections).
  So no spinor
  bundle occurs in the decomposition of $E$, and we are done.
\end{proof}

The next proposition is the main ingredient of the proof of our main
result.
Roughly speaking, it uses the lemmas of the previous section to 
describe the restriction to
7-dimensional quadrics of a complex orthogonal to our collection.
We will write $\bR \Phi$ the right-derived functors of a left-exact functor
$\Phi$, and similarly for the left-derived $\bL \Phi$ of a right-exact functor.
Given a subvariety $Z$ of $X$, we will denote by $i_Z$ the embedding of 
$Z$ in $X$.

\begin{prop} \label{U}
  Let $E$ be an object of the subcategory ${}^\perp\sD$ of $ \Db(X)$.
  Let $Q=Q_7$ be a smooth 7-dimensional quadric  in $X$.
  Then $\bL i_Q^*(E)$ is a direct sum of shifts of $\sO_{Q}(11)$.
\end{prop}

\begin{proof} 
  By Lemma \ref{3quadrics}, 
  there exist $s \in H^0(X,\cS)$ and $\sigma \in H^0(X,\cS_s)$ such that
  $Q$ is a component of $\sQ=\sQ_\sigma$,
  so that  $\sQ$ is the disjoint union of $Q$ and two
  additional 7-dimensional quadrics $Q'$ and $Q''$. 
  Denote by $\Sigma_\sQ$
  the sheaf on $\sQ$ that restricts to the spinor bundle on $Q$,
  $Q'$ and $Q''$, and set $j=i_\sQ$.

    Consider the following subcategories of $\Db(Q)$ and of $\Db(\sQ)$:
    \begin{align*}
      \sU & = \langle
      \sO_Q(5),\sO_Q(6),\Sigma(6),\sO_Q(7),\sO_Q(8),\sO_Q(9),\sO_Q(10) \rangle,\\
      \sV & = \langle
      \sO_\sQ(5),\sO_\sQ(6),\Sigma_\sQ(6),\sO_\sQ(7),\sO_\sQ(8),\sO_\sQ(9),\sO_\sQ(10) \rangle.
    \end{align*}
    What we have to prove is that $\bL i_Q^*(E)$ lies in ${}^\perp \sU$.
    Indeed, by \cite{kapranov:derived-homogeneous}, we have
    $\Db(Q) = \langle \sU ,\sO_Q(11) \rangle$,
    so that ${}^\perp \sU = \langle \sO_Q(11) \rangle$. Hence any 
    object of ${}^\perp \sU$ is a direct sum of shifts of $\sO_Q(11)$.
    
    To prove that $\bL i_Q^*(E)$ lies in ${}^\perp \sU$, it is enough
    to show that $\bL j^*(E)$ lies in ${}^\perp \sV$.
    In order to achieve this, since $E$ lies in ${}^\perp \sD$,
    it suffices to prove that, given a generator $v$ of
    $\sV$, the sheaf $\bR j_*(v) \cong j_*(v)$ lies in $\sD$.
    So we have to show that:
    \begin{align}
      \label{5-10}
      & j_*(\sO_\sQ(t)) \in \sD && \mbox{for $5\le t \le 10$,}  \\
      \label{sigma7}
      & j_*(\Sigma_\sQ(6)) \in \sD. 
    \end{align}
    
    To accomplish this task, 
    we first note that $\cS_{|\sQ} \cong \sO_\sQ \oplus \sO_\sQ(1) \oplus \Sigma_\sQ$.
    Then, we consider the Koszul complex of the global section $\sigma$ of
    $\cS_s$:
    \[
    0 \to \wedge^9 \cS_s^*\cong\sO_X(-5) \to \wedge^8 \cS_s^* \to \cdots \to \wedge^2 \cS_s^*
    \to \cS_s^* \to \sO_X \to j_*(\sO_\sQ) \to 0.
    \]
    It is now clear that, to prove \eqref{5-10} and \eqref{sigma7} it suffices to
    prove, respectively, that:
    \begin{align}
      \label{F(t)}
      & \wedge^p \cS_s^*(t) \in \sD && \mbox{for $0 \le p \le 9$,} &&
      \mbox{and for $5\le t \le 10$,} \\
      \label{S(t)}
      & \wedge^p \cS_s^* \ts \cS^*(7) \in \sD, && \mbox{for $0 \le p \le 9$.}
    \end{align}
    
    Let us first focus on \eqref{F(t)}.
    Taking exterior powers of \eqref{F}, and of its dual, we get:
    \[
    \wedge^p \cS_s^* \in \langle \wedge^p \cS^*, \wedge^{p-1} \cS^*, \ldots,
    \sO_X \rangle. 
    \]
    Note also that $\wedge^p \cS_s^* \cong \wedge^{9-p} \cS_s(-5)$.
    We deduce that \eqref{F(t)} holds as soon as:
   \[   \wedge^p \cS^*(t) \in \sD, \quad \mbox{and} \quad \wedge^p
      \cS^*(p+t-5) \in \sD, \quad \mbox{for $0 \le p \le 4$ and $5 \le
        t \le 10$}.
      \]
    This last fact is clear for $p=0,1$. Moreover, it follows from  Lemma \ref{wedge2}
  for $p=2$, Lemma \ref{wedge3} for $p=3$ and Lemma \ref{wedge4} for $p=4$.

    We now turn to the proof of
    \eqref{S(t)}.
    By the same argument as above, we reduce the statement to:
    \[
      \wedge^p \cS^*(7) \ts \cS^* \in \sD, \quad \mbox{and} \quad \wedge^p
      \cS^*(p+2) \ts \cS^* \in \sD, \quad \mbox{for $0 \le p \le 4$}.
      \]
    This is clear for $p=0$. For $p=1$, it amounts to:
    \[
    \wedge^2 \cS^*(7), \cS_2^*(7),\sO_X(7),
    \wedge^2 \cS^*(3), \cS_2^*(3),\sO_X(2) \in \sD,
    \]
    which follows from Lemma \ref{wedge2} and Lemma \ref{6term}.
    For $p=2,3,4$, we can apply  Corollary \ref{wedge-tensor}, and we are done.
\end{proof}

The last lemma that will be useful to us here is an analogue of
\cite[Theorem 3.2.1]{okonek-schneider-spindler}.

\begin{lem} \label{incidence}
  Fix a point $x \in X$,
  and let $E$ be an object of $\Db(X)$ such that, 
  for any 8-dimensional quadric $Q_8 \subset X$ through $x$, we
  have 
  \[
  \bL i_{Q_8}^* (E) \cong \bigoplus_k \sO^{r_k}_{Q_8}[-k]
  \]
 for some integers $r_k$. Then:
  \[E \cong \bigoplus_{k}\sO_X^{r_k}[-k]. \]
\end{lem}

\begin{proof}
The argument follows the 
proof given in \cite{okonek-schneider-spindler} that any bundle on the projective plane, which is trivial
on any line through a given point, must be trivial. This proof applies 
almost verbatim since the Cayley plane and its dual define a kind of plane projective
geometry. Let:
$$\cB_x=\{(y,\ell)\in \OO\PP^2\times \OO\check{\PP}^2, \quad x,y\in Q_\ell\}.$$
The second projection $q : \cB_x\to\OO\check{\PP}^2$ maps $\cB_x$ to $\check{Q}_x$, 
and makes of $\cB_x$ a locally trivial fibration in smooth
8-dimensional quadrics  over $\check{Q}_x$.
In particular, $\cB_x$ is smooth of dimension $16$. The first projection $p : 
\cB_x\to\OO\PP^2$ is birational since a general point of the Cayley plane 
belongs to a unique $\OO$-line $Q_\ell$ passing through $x$. 
Recall also that $q$ has a
section $\xi$ defined by $\xi(\ell)=(x,\ell)$ so $p\circ \xi$ is the
constant map $x$, while of course $q \circ \xi$ is the identity.
\[
\xymatrix@-3ex{
& \cB_x \ar_-{p}[dl] \ar_-{q}[dr]& \\
X && \ar@/_/_-{\xi}[ul] \check{Q}_x 
}
\] 

Just as in \cite{okonek-schneider-spindler}, we would like to find a
complex $F$ in $\Db(\check{Q}_x)$ such that:
\begin{equation}
  \label{E-F}
\bL q^*(F) \cong \bL p^*(E).  
\end{equation}

Let us first check that this implies our statement. 
Note that, by \eqref{E-F}, we have:
\[
F \cong \bL (q \circ \xi)^* (F) \cong  \bL \xi^* (\bL q ^* (F)) \cong 
\bL \xi^* (\bL p ^* (E)) \cong \bL (p \circ \xi)^*(E) \cong
\bigoplus_k \sO_{\check{Q}_x}^{r_k}[-k],
\]
where the last isomorphism holds since $p \circ \xi$ has constant
value $x$, so $\bL (p \circ \xi)^*(E)$ is obtained as $(p \circ \xi)^*(\bL
i^*_{x}(E))$, and clearly $\bL i^*_{x}(E) \cong \bigoplus_k \CC_x^{r_k}[-k],$
by restriction from $Q_8$, since all the $\bL^k i^*_{Q_8}(E)$ are free.
Hence, we get:
\[
\bL p^*(E) \cong \bL q^* (F) \cong \bigoplus_k \sO_{\cB_x}^{r_k}[-k],
\]
and, in turn:
\[
E \cong \bR  p _* (\bL p^*(E)) \cong \bigoplus_k \bR
p_*(\sO_{\cB_x})^{r_k}[-k] \cong \bigoplus_k \sO_{X}^{r_k}[-k],
\]
where the last isomorphism is clear, since $p$ being birational, 
$\bR p_*(\sO_{\cB_x}) \cong \sO_X$.

\medskip

It remains to prove \eqref{E-F}.
We set:
\[
F = \bR q_* (\bL p^*(E)),
\]
and we have a natural morphism in $\Db(\cB_x)$:
\[
\varphi : \bL q^*(F) = \bL q^*(\bR q_* (\bL p^*(E))) \to \bL p^*(E).
\]
We would like to prove that, for any $\ell \in \check{Q}_x$, 
this morphism restricts to an isomorphism in the derived category
$\Db(Q_\ell)$ of the fibre $Q_\ell$ of $q$ over $\ell$, and this will
finish the proof since these fibres cover $\cB_x$.
Denote by $\alpha$ the embedding of $\{\ell\}$ into $\check{Q}_x$, by $\beta$
the restriction of $q$ to $Q_\ell \to \{\ell \}$ and by $\gamma$ the
embedding $Q_\ell \to \cB_x$ so that $p \circ \gamma = i_{Q_\ell}$.
We have to prove that $\bL \gamma^*(\varphi)$ is an isomorphism.
We write $\bL \gamma^*(\varphi)$ as:
\begin{equation}
  \label{phi}   \bL \gamma^*\bL q^*\bR q_* \bL p^*E \to \bL \gamma^*\bL p^*E.
\end{equation}
On the right-hand-side, we have natural isomorphisms:
\[  \bL \gamma^*\bL p^*E \cong \bL (p \circ \gamma)^*E \cong
  \bL i_{Q_\ell}^*E \cong \bigoplus_{k}\sO_{Q_\ell}^{r_k}[-k].
\]
On the left-hand-side of \eqref{phi}, we have natural isomorphisms:
\[
  \bL \gamma^*\bL q^*\bR q_* \bL p^*E \cong
  \bL \beta^* \bL \alpha^* \bR q_* \bL p^* E \cong 
  \bL \beta^* \bR \beta_* \bL \gamma^* \bL p^* E \cong 
  \bL \beta^* \bR \beta_*  \bL i_{Q_\ell}^*E,
\]
where the middle one is given by smooth base-change.
So $\bL \gamma^*(\varphi)$ is the natural map:
\[
\bL \beta^* \bR \beta_*  \bL i_{Q_\ell}^*E \to \bL i_{Q_\ell}^*E,
\]
which is clearly an isomorphism since 
$\bL i_{Q_\ell}^*E \cong \bigoplus_{k}\sO_{Q_\ell}^{r_k}[-k]$.
\end{proof}

\section{Proof of the main result}

\label{3}

Here we prove that our subcategory $\sD$ generates the whole $\Db(X)$.
We have to show that ${}^\perp \sD = 0$.
The idea, inspired on an argument appearing in
\cite{bondal-orlov:semiorthogonal-arxiv},
is to restrict $E$ to the family of 8-dimensional quadrics in $X$ through a
given point.
So let $E$ be a complex of coherent sheaves on $X$, lying in
  ${}^\perp \sD = 0$, and let us prove that $E=0$.

We let $x$ be a point of $X$ such that, for all $k$, $\cH^k(E)$ is locally free around
$x$. We have:
\[
\bL i_x^* (E) \cong \bigoplus_k  \bL^k i_x^*(E)[-k]
\cong  \bigoplus_k \cH^k(E)_{|x}[-k] \cong  \bigoplus_k \CC_x^{r_k}[-k],
\]
for some integers $r_k$, with $r_k \ne 0$ for finitely many $k$'s.

\medskip

Let us first use 7-dimensional quadrics. So let $Q_7$ be a smooth
7-dimensional quadric contained in $X$ and passing through $x$.
In view of Proposition \ref{U}, $\bL i^*_{Q_7} (E)$ is the direct sum 
  of shifts of $\sO_{Q_7}(11)$, so that:
  \[
  \bL i^*_{Q_7}(E(-11)) \cong \bigoplus_k \bL^k i^*_{Q_7}(E(-11))[-k]
  \cong \bigoplus_k \sO^{t_k}_{Q_7}[-k],
  \]
  for some (finite) sequence of integers $t_k$.
  All of these cohomology sheaves are free, so restricting from $Q_7$ to $x$ each of the
  $\bL^k i^*_{Q_7}(E(-11))$ we must get $\bL^k i^*_{x}(E(-11)) \cong
  \CC_x^{r_k}$.
  Hence $t_k=r_k$ for each $k$.

  Further, note that, given another smooth 7-dimensional
  quadric $Q_7' \subset X$, meeting $Q_7$, we will have:
  \[
  \bL i^*_{Q_7'}(E(-11)) \cong \bigoplus_k \bL^k i^*_{Q_7'}(E(-11))[-k]
  \cong \bigoplus_k \sO^{r_k}_{Q_7'}[-k],
  \]
  for the same sequence of
  integers $r_k$.
  Indeed, the rank of the restricted free sheaves to any
  point of $Q_7 \cap Q_7'$ must
  agree.

  \medskip
  Now let us move to 8-dimensional quadrics.
  For any smooth such quadric $Q_8 \subset X$ containing $x$, we consider the (finitely
  many) non-zero sheaves
  $\bL^k i_{Q_8}^*(E)$.
  We get finitely many proper subschemes of $Q_8$ as torsion loci of
  these sheaves, and we denote by $\tau(\bL^k i_{Q_8}^*(E))$ the
  torsion locus of each $\bL^k i_{Q_8}^*(E)$.
  So, we may find a family of smooth hyperplane sections $Q_7$ of $Q_8$,
  covering the whole $Q_8$, such
  that, for any given $k$, $\tau(\bL^k i_{Q_8}^*(E))$  does not contain
  $Q_7$. Let $f$ be the linear form cutting $Q_7$ in $Q_8$.
  We have a long exact sequence:
  \[
  \cdots \to \bL^{k-1} i^*_{Q_7} (E) \to
  \bL^k i^*_{Q_8} (E(-1)) \xr{f}  \bL^k i^*_{Q_8} (E)  \to
  \bL^k i^*_{Q_7} (E)  \to \bL^{k+1} i^*_{Q_8} (E(-1))  \to \cdots 
  \]
  Since $\tau(\bL^k i_{Q_8}^*(E)) \not \supset Q_7$ for all $k$, all
  the maps $f=\bL^k i_{Q_8}^*(f)$ in the above sequence are injective.
  So $\bL^k i^*_{Q_7} (E) \cong (\bL^k i^*_{Q_8} (E))_{| Q_7}$, and we
  may split the above long sequence into short exact sequences:
  \begin{equation}
    \label{restriction}
    0  \to
    \bL^k i^*_{Q_8} (E(-1)) \xr{f}  \bL^k i^*_{Q_8} (E)  \to
    \bL^k i^*_{Q_7} (E) \to 0.
  \end{equation}
  Using Proposition \ref{U}, this means that, for any $k$, the sheaf $\bL^k i^*_{Q_8}(E(-11))$ restricts to
  $Q_7$ as $\sO_{Q_7}^{r_k}$.

  We have a similar exact sequence for any other $Q_7'$ in our covering family of hyperplane
  sections of $Q_8$.
  Since all these sections $Q_7'$ obviously meet $Q_7$, 
  using again \eqref{restriction} we get 
  that, for any $k$, the sheaf $\bL^k i^*_{Q_8}(E(-11))$ restricts to
  $Q_7'$ as $\sO_{Q'_7}^{r_k}$ (for the same integers $r_k$).
  In particular $\bL^k i^*_{Q_8}(E(-11))$ is locally free of rank
  $r_k$.
  Lemma \ref{Q2} now gives $\bL^k i^*_{Q_8}(E(-11)) \cong \sO_{Q_8}^{r_k}$.

  \medskip
  We have thus, for all $k$, an isomorphism  $\bL^k i^*_{Q_8}(E(-11)) \cong \sO_{Q_8}^{r_k}$.
  We would like to conclude that $E(-11)$ is itself isomorphic to
  $\bigoplus_k \sO_{X}^{r_k}[-k]$.  
  To achieve this, consider 
  any other smooth 8-dimensional
  quadric $Q'_8$  contained in $X$.
  By the same argument as above, we get that, given any $k$, the sheaf
  $\bL ^k i^*_{Q_8'} (E(-11))$ is free.
  The rank of this free sheaf must be $r_k$, since $Q_8$ and
  $Q'_8$ meet at $x$, and restricting
  $\bL ^k i^*_{Q_8'} (E(-11))$ to $x$ must give $\bL^k i^*_{x}
  (E(-11)) \cong \CC^{r_k}$.

  \medskip
  We are now in position to apply Lemma \ref{incidence} to obtain an isomorphism
  $E(-11) \cong \bigoplus_k \sO_{X}^{r_k}[-k]$, so that $E \cong \bigoplus_k \sO_{X}^{r_k}(11)[-k]$.
  But $E$ belongs to ${}^\perp \sD$, in particular
  $\Ext^i_X(E,\sO_X(11))=0$ for all $i$. Therefore $r_k=0$ for all $k$,
  so $E=0$.
  The proof is now complete.

\bibliographystyle{alpha}
\bibliography{category-cayley}

\def\cprime{$'$} \def\cprime{$'$} \def\cprime{$'$} \def\cprime{$'$}
  \def\cprime{$'$} \def\cprime{$'$} \def\cprime{$'$}
\begin{thebibliography}{Man11}

\bibitem[BO95]{bondal-orlov:semiorthogonal-arxiv}
Alexei~I. {Bondal} and Dmitri~O. {Orlov}.
\newblock {Semiorthogonal decomposition for algebraic varieties}.
\newblock {\em ArXiv e-prints}, June 1995.

\bibitem[IM11]{iliev-manivel:78}
Atanas {Iliev} and Laurent {Manivel}.
\newblock {On cubic hypersurfaces of dimension seven and eight}.
\newblock {\em ArXiv e-prints}, February 2011.

\bibitem[Kap88]{kapranov:derived-homogeneous}
Mikhail~M. Kapranov.
\newblock On the derived categories of coherent sheaves on some homogeneous
  spaces.
\newblock {\em Invent. Math.}, 92(3):479--508, 1988.

\bibitem[Kn{\"o}87]{knorrer:ACM}
Horst Kn{\"o}rrer.
\newblock Cohen-{M}acaulay modules on hypersurface singularities. {I}.
\newblock {\em Invent. Math.}, 88(1):153--164, 1987.

\bibitem[Kuz07]{kuznetsov:hpd}
Alexander~G. Kuznetsov.
\newblock Homological projective duality.
\newblock {\em Publ. Math. Inst. Hautes \'Etudes Sci.}, (105):157--220, 2007.

\bibitem[LM03]{landsberg-manivel:projective}
Joseph~M. Landsberg and Laurent Manivel.
\newblock On the projective geometry of rational homogeneous varieties.
\newblock {\em Comment. Math. Helv.}, 78(1):65--100, 2003.

\bibitem[LM04]{landsberg-manivel:series}
Joseph~M. Landsberg and Laurent Manivel.
\newblock Series of {L}ie groups.
\newblock {\em Michigan Math. J.}, 52(2):453--479, 2004.

\bibitem[Man11]{manivel:cayley}
Laurent Manivel.
\newblock On the derived category of the {C}ayley plane.
\newblock {\em J. Algebra}, 330:177--187, 2011.

\bibitem[OSS80]{okonek-schneider-spindler}
Christian Okonek, Michael Schneider, and Heinz Spindler.
\newblock {\em Vector bundles on complex projective spaces}, volume~3 of {\em
  Progress in Mathematics}.
\newblock Birkh\"auser Boston, Mass., 1980.

\bibitem[Ott88]{ottaviani:spinor}
Giorgio Ottaviani.
\newblock Spinor bundles on quadrics.
\newblock {\em Trans. Amer. Math. Soc.}, 307(1):301--316, 1988.

\end{thebibliography}

\end{document}